\documentclass[11pt]{article}

\usepackage{amsthm,amssymb,amsmath}

\usepackage{hyperref}

\theoremstyle{plain}
\newtheorem{theorem}{Theorem}

\newtheorem{lemma}{Lemma}

\theoremstyle{definition}
\newtheorem{definition}[theorem]{Definition}

\theoremstyle{remark}
\newtheorem{remark}[theorem]{Remark}

\advance \topmargin by -\headheight
\advance \topmargin by -\headsep
\author{Jean-Paul Allouche \\
CNRS, IMJ-PRG \\
Sorbonne Universit\'e \\
4 Place Jussieu \\
F-75252 Paris Cedex 05, France \\
{\tt jean-paul.allouche@imj-prg.fr}
\and
Guo-Niu Han \\
CNRS, IRMA \\
Universit\'e de Strasbourg \\
7 rue Ren\'e Descartes \\
F-67084 Strasbourg, France \\
{\tt guoniu.han@unistra.fr}
\and Harald Niederreiter \\
Johann Radon Institute \\
RICAM, Austrian Academy of Sciences \\
Altenberger Stra{\ss}e 69 \\
A-4040 Linz, Austria \\
{\tt ghnied@gmail.com}
}

\title{Perfect linear complexity profile and Apwenian sequences}

\date{ }

\begin{document}

\maketitle

\begin{abstract}
Sequences with {\em perfect linear complexity profile} were defined more than thirty years ago
in the study of measures of randomness for binary sequences. More recently {\em apwenian 
sequences}, first with values $\pm 1$ , then with values in $\{0, 1\}$, were introduced in the study 
of Hankel determinants of automatic sequences. We explain that these two families of sequences
are the same up to indexing, and give consequences and questions that this implies. We hope that
this will help gathering two distinct communities of researchers.

\medskip

\noindent
{\bf Keywords}: Perfect linear complexity profile, generalized Rueppel sequences, Apwenian 
sequences, continued fractions with partial quotients with bounded degree, automatic sequences.

\medskip

\noindent
{\bf MSC}: 11K45, 11K50, 11J70, 11B85, 11T71, 94A55.

\end{abstract}

\bigskip

One of the intense pleasures in mathematical research is to discover a link between two fields that
either did not seem immediately related or were studied from two distinct points of view unaware of 
each other. The first author, reading the paper \cite{Guo-Han-Wu}, saw a relation 
satisfied by the so-called ``$0,1$-apwenian sequences'' $(c_n)_{n \geq 0}$, namely
$$
c_0 = 1, \ \text{and} \ \forall n \geq 0, \ c_n \equiv c_{2n+1} + c_{2n+2} \ \bmod 2.
$$
This relation sounded familiar: he vaguely remembered a talk at SETA 98 where a similar relation
was discussed, but was not able to find a reference, though he had the name of the third 
author in mind. And indeed the latter indicated to him two papers of his with a similar 
relation \cite{Niederreiter1, Niederreiter2}. The second of these papers contains {\em inter alia} a 
simple proof of a result initially due to Wang and Massey \cite{Wang-Massey} which displays the 
following relation for the so-called ``PLCP sequences'' $(s_i)_{i \geq 1}$:
$$
s_1 = 1, \ \text{and} \  \forall i \geq 1, \ s_{2i+1} \equiv s_{2i} + s_i \ \bmod 2.
$$
Note that this relation already occurs in a 1977-paper of Baum and Sweet 
\cite[p.~574]{Baum-Sweet2}. As far as SETA 98 is concerned, the property above indeed occurs 
in \cite[Theorem~3]{Kohel-Ling-Xing}. 

\medskip

In this paper we propose to describe the links between PLCP sequences and apwenian sequences:  
unexpected connections arise from the two properties of $0,1$-sequences displayed above.

\section{PLCP sequences}

The construction of sequences with {\it almost perfect linear complexity profile} was motivated
by the search for pseudorandom sequences having reasonable properties of unpredictability
and randomness. We recall two definitions and some properties (see, e.g. \cite{Niederreiter2}).

\begin{definition}
A sequence $(s_n)_{n \geq 1}$ of elements in a field $F$ is called a $k$-th order {\em shift-register 
sequence} if there exist constants $a_0, a_1, \ldots, a_{k-1}$ in $F$ such that, for all $ i \geq 1$
$$
s_{i+k} + a_{k-1} s_{i+k-1} + \ldots + a_1 s_{i+1} + a_0 s_i = 0.
$$
\end{definition}

\begin{remark}
Shift-register sequences are also called {\em sequences satisfying a linear recurrence relation}.
Also note that $(s_n)_{n \geq 1}$  is a shift-register sequence if and only if the formal power
series $\sum s_n X^n$ is rational (i.e., can be obtained as the quotient of two polynomials).
\end{remark}

\begin{definition}
The {\em linear complexity} $L(n)$ of a sequence $(s_n)_{n \geq 1}$ of elements in a field $F$ is 
defined as the least $k$ such that $s_1, s_2, \ldots, s_n$ are the first $n$ terms of a $k$-th order
shift-register sequence. In the case where the first $n$ terms of $(s_n)$ are $0$, $L(n)$ is defined 
by $L(n)=0$. The sequence $(L(n))_{n \geq 1}$ is called the {\em linear complexity profile} of the 
sequence $(s_n)_{n \geq 1}$.
\end{definition}

\begin{remark}
One clearly has $0 \leq L(n) \leq n$ and $L(n) \leq L(n+1)$.
\end{remark}

Actually the linear complexity of a sequence $(s_n)_{n \geq 1}$ is related to the continued fraction
expansion of the formal Laurent series $\sum_{n \geq 1} s_n t^{-n}$. Recall that formal Laurent 
series can be expanded into continued fractions similarly to the real case, where ${\mathbb R}$
is replaced by $F((t^{-1}))$ and ${\mathbb N}$ by $F[t]$. To the best of our knowledge E. Artin 
(in his thesis) was the first author who defined and studied these continued fractions (see \cite{Artin}). 

\begin{theorem}[{\rm Theorem~1 in \cite{Niederreiter2}}]
Let $P_j/Q_j$ be the convergents of the series $\sum_{n \geq 1} s_n t^{-n}$ and $(L(n))_{n \geq 1}$
be the linear complexity profile of the sequence $(s_n)_{n \geq 1}$. Then
$$
L(n) = \deg Q_j
$$
where $j$ is the unique integer defined by
$$
\deg Q_{j-1} + \deg Q_j  \leq  n <  \deg Q_j + \deg Q_{j+1}.
$$
\end{theorem}

Now what is the ``typical'' linear complexity profile for a sequence? Since Rueppel \cite{Rueppel}
proved that the linear complexity profile of a random binary sequence is $n/2 + O(1)$ (where 
actually $0 \leq O(1) \leq 5/18$), ``good'' sequences are sequences whose linear complexity profile 
is as close to $n/2$ as possible. For more precise results on the typical linear complexity profile for
sequences with values in a finite field, see \cite{NiedereiterProba}. Also see \cite{SXJ}.
The following definition was adopted.

\begin{definition}
A sequence $(s_n)_{n \geq 1}$ of elements of a field F is said to have a 
{\em perfect linear complexity profile (PLCP)} if for all $n \geq 1$ one has 
$L(n) = \lfloor \frac{n+1}{2} \rfloor = \lceil \frac{n}{2} \rceil$.
(Such sequences are also called {\it $1$-perfect}.)
\end{definition}

\begin{theorem}[{\rm Theorem~2 in \cite{Niederreiter2}}]\label{Th2Nie2}
The sequence $s_1, s_2, \ldots$ has a PLCP if and only if its generating function 
$\sum_{i \geq 1} s_i t^{-i}$ is irrational and has all partial quotients of degree $1$ in its continued 
fraction expansion.
\end{theorem}

\section{Apwenian sequences}

Let us first recall that the Hankel determinants of a sequence $(a_n)_{n \geq 0}$ are defined by
$$
H_n = 
\left|
\begin{array}{lllll}
&a_0  &a_1 &\cdots &a_{n-1} \\
&a_1  &a_2 &\cdots  &a_n  \\
&\vdots &\vdots &\ddots &\vdots \\
&a_{n-1} &a_n &\cdots &a_{2n-2} \\
\end{array}
\right|
= (\det(a_{i+j}))_{0 \leq i, j \leq n-1}.
$$

Apwenian sequences (in fact, a variant, namely $\pm$-Apwenian sequences) were defined in 
\cite{Fu-Han} and criteria for a sequence to be apwenian were given in \cite{Guo-Han-Wu}. The 
origin of these sequences is the paper \cite{APWW} where ---thanks to Z.-X. Wen and Z.-Y. Wen 
who discovered and studied sixteen simultaneous recurrence formulas--- it was stated that the 
Hankel determinants of the $\pm 1$ Thue-Morse sequence satisfy, for $n \geq 1$ the congruence 
$H_n/2^{n-1} \equiv 1 \bmod 2$, which of course implies that $H_n$ is not zero. In particular, Hankel 
determinants were proved to form a $2$-dimensional $2$-automatic sequence and the frequency of 
nonzero Hankel determinants was proved to be positive, which implies properties of non-repetition in 
the Thue-Morse sequence and the existence of certain Pad\'e approximants. The references in 
\cite{Guo-Han-Wu} point to several papers studying Hankel determinants related to automatic 
sequences, yielding in particular irrationality measures for certain automatic real numbers.

\bigskip

The case of $0,1$-sequences was studied in \cite{Guo-Han-Wu}, where the following definition
was given

\begin{definition}
Let ${\mathbf c} = (c_n)_{n \geq 0}$ be a sequence with values in $\{0, 1\}$, such that $c_0 = 1$. 
Let $(H_n)_n$ be the sequence of its Hankel determinants. The sequence ${\mathbf c}$ is said to 
be {\em apwenian} if for all $n \geq 0$ one has
$$
H_n \equiv 1 \bmod 2.
$$ 
\end{definition}

\section{A single theorem gathering previous results}

In this section we state a theorem that puts together seemingly unrelated previous results.

\begin{theorem}\label{equivalences}
Let $(s_n)_{n \geq 1}$ be a sequence with values in ${\mathbb F}_2$ with $s_1 = 1$.
Let $(c_n)_{n \geq 0}$ be the sequence defined by $c_n = s_{n+1}$ for all $n \geq 0$.
(Note that $c_0 = 1$.) Then the following properties are equivalent.

\begin{itemize}

\item[ ]{{\rm (i)}}  The sequence $(s_n)_{n \geq 1}$ has a perfect linear complexity profile (PLCP).

\item[ ]{{\rm (ii)}} The Laurent series $\sum_{n \geq 1} s_n t^{-n}$ is irrational, and all the partial
quotients in its (usual) continued fraction expansion have degree $1$.

\item[ ]{{\rm (iii)}} For all $n \geq 1$ one has $s_{2n+1} = s_{2n} + s_n$.

\item[ ]{{\rm (iv)}} For all $n \geq 0$ one has $c_{2n+2} = c_{2n+1} + c_n$.

\item[ ]{{\rm (v)}} The sequence $(c_n)_{n \geq 0}$ is apwenian.

\end{itemize}
\end{theorem}

\begin{proof} The equivalence of  (iii) and (iv) is straightforward ---just a shift of indices.
But this is {\em the} point that links the two fields of perfect linear complexity and of 
nonzero Hankel determinants.

The equivalence of (i) and (ii) was proved in \cite[Theorem~2]{Niederreiter2} (also see
Theorem~\ref{Th2Nie2}).

The equivalence of (i) and (iii) was proved in \cite{Wang-Massey}, and a simpler proof was given
in \cite{Niederreiter2} using an idea of \cite{Taussat}.

The equivalence of (iv) and (v) was proved in \cite[Theorem~1.5]{Guo-Han-Wu}.

\end{proof}

\begin{remark}\label{examples1}
Several examples of PLCP or apwenian sequences can be found in the literature. We cite some
of them.

\begin{itemize}

\item[ ] {$*$} The first two examples are the two (simple) Rueppel sequences. Recall that these 
sequences, say $(r_n)_{n \geq 1}$ and $(s_n)_{n \geq 1}$, are defined, respectively, as the 
characteristic function of the powers of $2$ and the characteristic function of the integers of the form 
$(2^k - 1)$. Thus $r_1 = 1$, $r_{2n+1} = 0$ for all $n \geq 1$ and $r_{2n} = r_n$ for all $n \geq 1$. 
Similarly $s_1 = 1$, $s_{2n+1} = s_n$ for all $n \geq 1$, and $s_{2n} = 0$ for all $n \geq 1$. 

\item[ ]{$*$} Generalized Rueppel sequences were defined in \cite{Niederreiter1}
(they are also mentioned, e.g., in \cite{xing-niederreiter}): first take a sequence of integers 
$n_1, n_2, \ldots$ such that $n_1 = 1$ and $n_{h+1} = 2 n_h + c_h$ for $h = 1, 2, \ldots$ 
where $c_h \in \{0, 1\}$; the generalized Rueppel sequence (associated with $(c_h)_h$) 
is the sequence $(s_i)_{i \geq 1}$ defined by $s_i = 1$ if $i = n_h$ for some $h \geq 1$, 
and $s_i = 0$ otherwise.

\item[ ]{$*$} Sequences with PLCP constructed from curves over finite fields can be found in 
\cite{XNLD} (also see the references therein).

\item[ ]{$*$} Examples of apwenian and non-apwenian sequences can be found in \cite{Guo-Han-Wu}.

\end{itemize}

Actually, generalized Rueppel sequences that are defined above can be viewed as generated by a 
map from the set of all binary sequences to the set of PLCP sequences or to the set of apwenian  
sequences after a small modification. There are at least three ``natural'' (or ``simple''?) similar 
maps from the set of binary sequences to the set of apwenian sequences or to the set of PLCP 
sequences, namely $\varphi_i : \mathbf{b} = (b_n)_n \mapsto \mathbf{a} = (a_n)_n$ as follows.
 
 \begin{itemize}
 
\item[ ]{$*$} The first map $\varphi_1$ is the bijection defined by $\mathbf{a}=\varphi_1(\mathbf{b})$, 
with, in terms of (Artin) continued fraction 
\begin{equation*}
        \sum_{n\geq 0} {a_n}t^{-n-1}=
\cfrac{1}{t+b_0 +\cfrac{1}{t+b_1+\cfrac{1}{t+b_2+\cfrac{1}{\ddots}}}}
\end{equation*}
or in terms of Jacobi continued fraction
\begin{equation*}
        \sum_{n\geq 0} {a_n}x^n=
\cfrac{1}{1+b_0x+\cfrac{x^{2}}{1+b_1x+\cfrac{x^2}{1+b_2x+\cfrac{x^2}{\ddots}}}}\cdot
\end{equation*}

\item[ ]{$*$} The second map $\varphi_2$ is a bijection defined by $\mathbf{a}=\varphi_2(\mathbf{b})$ 
that arises from a ``selector''. Namely the apwenian sequence $\mathbf{a}$ satisfies
$$
a_0=1,\quad a_n=a_{2n+1}+a_{2n+2}\ (n\geq 0).
$$
Since $a_0=1=a_1+a_2$, we have two choices for $a_1$ and $a_2$: (i) $a_1=0$ and $a_2=1$, 
(ii) $a_1=1$ and $a_2=0$. We select one of the two choices according to the value of $b_0$. If 
$b_0=0$, we select (i); if $b_0=1$, we select (ii). In general, if $a_n=1$, we construct $a_{2n+1}$ 
and $a_{2n+2}$ in the same way. We define $a_{2n+1}=b_n$, $a_{2n+2}=1+b_n$. If $a_n=0$, 
$0=a_{2n+1}+a_{2n+2}$. We have two choices: (iii) $a_{2n+1}=a_{2n+2}=0$, 
(iv) $a_{2n+1}=a_{2n+2}=1$. We will select one of the two choices from the value of $b_n$. If 
$b_n=0$, we select (iii); if $b_n=1$, we select (iv). We define $a_{2n+1}=a_{2n+2}=b_n$. This
process is of course exactly defining inductively the sequence by $a_0 = 1$ and, if $a_n$ is known, 
then $a_{2n+1} = b_n$ and $a_{2n+2} = a_n + b_n$.

\item[ ]{$*$} The third map $\mathbf{a}=\varphi_3(\mathbf{b})$ has been defined above (we 
replace $(c_h)$ with $(b_h)$ here). First we define a sequence $n_0, n_1, n_2, \ldots$ by
$$
n_0=1, \quad n_{h+1} = 2 n_h + b_h \ (h\geq 0).
$$
Then we define $a_i=1$ if $i=n_h$ for some $h\geq 0$, and $a_i=0$ otherwise.
Note that this map is not a bijection as the example of the fixed point of $1 \to 10$, $0 \to 11$
(the period-doubling sequence) shows: it is not a generalized Rueppel sequence because it 
contains infinitely many occurrences of the block $111$.
 
\end{itemize}
 
An interesting question is to study which general properties of sequence ${\mathbf b}$ are kept 
when applying $\varphi_i$, $i = 1,2,3$. Typically, if ${\mathbf b}$ has some kind of ``regularity'',
is it also the case for ${\mathbf a}$? See in particular Remark~\ref{examples2} below.

\end{remark}

\section{Automaticity of sequences having PLCP or being apwenian}

Is it possible to impose extra ``regularity'' conditions on apwenian-PLCP sequences?
In particular are there {\em automatic} or {\em morphic} such sequences? (For more about these
notions, the reader can consult \cite{Allouche-Shallit}.) This question was asked in \cite{Guo-Han-Wu},
where it was proved that the only $0,1$-apwenian sequence that is a fixed point of a uniform 
morphism is the period-doubling sequence (fixed point of $1 \to 10$, $0 \to 11$). Restricting to 
$2$-uniform morphisms, one cask a general question: are there $0,1$-apwenian sequences that 
are $2$-automatic? We completely characterize these sequences below. To begin with, we give
an easy lemma. 

\begin{lemma}
Let $f(t) = t + a_2 t^2 + a_3 t^3 + \ldots$ be a formal power series on ${\mathbb F}_2$. Then there
exist unique formal power series $u \in 1 + {\mathbb F}_2[[t]]$ and $v \in t{\mathbb F}_2[[t]]$ such 
that $f(t) = v^2(t) + tu^2(t)$. Furthermore one has $u^2 = f'$ and $v^2 = f + t f'$. 
\end{lemma}

\begin{proof}
Decomposing $f$ into its ``even'' and ``odd'' parts, we have (recall that $x^2 = x$ for any 
$x \in {\mathbb F}_2$):
$$
\begin{aligned}
f(t) &= (a_2 t^2 + a_4 t^4 + \ldots) + t (1 + a_3 t^2 + a_5 t^4 + \ldots) \\
&= (a_2 t + a_4 t^2 + \ldots)^2 + t (1 + a_3 t  + a_5 t^2 + \ldots)^2.
\end{aligned}
$$
This gives the existence of $v(t) = a_2 t + a_4 t^2 + \ldots$ and $u(t) = 1 + a_3 t  + a_5 t^2 + \ldots$
such that $f(t) = v^2(t) + tu^2(t)$.
Conversely, if $f$ as above satisfies $f(t) = v^2(t) + tu^2(t)$, then the uniqueness of $u$ and $v$
is a consequence of the uniqueness of the decomposition of $f$ into its odd and even parts.
The last assertion is clear (and it could be used to give another proof of the first assertion).
\end{proof}

We now address the question of which $0,1$-apwenian sequences are automatic. 
Our main tools are first the Christol (and Christol-Kamae-Mend\`es France-Rauzy) theorem 
(see, e.g., \cite{Allouche-Shallit}) which asserts that a sequence $(c_n)_n$ with coefficients in 
${\mathbb F}_2$ is $2$-automatic if and only if the formal power series $\sum c_n t^n$ in 
${\mathbb F}_2[[t]]$ is algebraic over ${\mathbb F}_2(t)$, and second the following result in 
\cite{Kohel-Ling-Xing}. (Note that this result is the same as \cite[Theorem~1]{Baum-Sweet2}.)

\bigskip

\begin{theorem}[{\rm \cite[Theorem~1]{Baum-Sweet2} and \cite[Corollary~2]{Kohel-Ling-Xing}}]
{\it The binary series $f(t) = t + a_2 t^2 + a_3 t^3 + ...$ has a PLCP if and only if the series
$u(t)$ and $v(t)$ defined above satisfy $v^2 +v =1+ u + t u^2$. (In other words, $v$ is
the unique root in $t {\mathbb F}_2[[t]]$ of $v^2 +v =1+ u + t u^2$.)}
\end{theorem}

It is easy to deduce from the result above the following theorem (which was proved in 
\cite{Baum-Sweet2} and for which we give a proof in terms of $u$ and $v$ above, and a second
proof using automatic sequences).

\begin{theorem}\label{alg}
A series $f(t) = t + a_2 t^2 + a_3 t^3 + ... \in {\mathbb F}_2[[t]]$ which has a PLCP is algebraic 
over ${\mathbb F}_2(t)$ if and only if it can be written $f(t) = v^2+t u^2$, with $u$ any series 
in $1+t {\mathbb F}_2[[t]]$ algebraic over ${\mathbb F}_2(t)$ and $v$ the root of 
$v^2 +v =1+ u + t u^2$ lying in $t {\mathbb F}_2[[t]]$.
\end{theorem}

\begin{proof}
The condition is clearly sufficient: if $u$ is algebraic over ${\mathbb F}_2(t)$, then $v$ is also
algebraic since $v^2 +v =1+ u + t u^2$, hence $f = v^2+t u^2$ is algebraic. Conversely it is
easy to see that if $f = v^2+t u^2$ is algebraic over ${\mathbb F}_2(t)$, then so is $f'$ (write a
minimal algebraic equation for $f$ and take the derivative: this yields that $f'$ belongs to  
${\mathbb F}_2(t)(f)$, hence is algebraic). But $f' = u^2$. Thus $u$ is algebraic, hence $v$ is 
algebraic. This implies the algebraicity of $f$.\end{proof}

\noindent
{\it Second proof.} 
It is also possible to prove Theorem~\ref{alg} by looking at the $2$-kernel of $f$. Recall that
the $2$-kernel of the sequence $(c_n)_{n \geq 0}$ is the set of subsequences 
$$
\{(c_{2^k n + j})_{n \geq 0}, \ k \geq 0, \ j \in [0, 2^k-1]\}.
$$
Note that the $2$-kernel of any sequence is the smallest set of subsequences of this sequence
that is stable by the decimation operators (sometimes called Cartier operators): 
$$
T_0: \ (z_n)_{n \geq 0} \to (z_{2n})_{n \geq 0} \ \text{and} \ 
T_1: \ (z_n)_{n \geq 0} \to (z_{2n+1})_{n \geq 0}.
$$
Also recall that a sequence is $2$-automatic if and only if its $2$-kernel is a finite set (see, e.g., 
\cite{Allouche-Shallit}) and that all computations on sequences below are done modulo $2$.

\medskip

Now, if $(a_n)_{n \geq 1}$ is a $2$-automatic sequence with a PLCP, with $a_1 = 1$, one has that 
$(a_{2n+1})_{n \geq 0}$ is also $2$-automatic. Define $(u_n)_{n \geq 0}$ and $(v_n)_{n \geq 0}$ 
by $u_n =  a_{2n+1}$ for all $n \geq 0$ and $v_n =  a_{2n}$ for all $n \geq 1$. Since $(a_n)$ has 
a PLCP, we have that $v_n = a_{2n} = a_n + a_{2n+1} = a_n + u_n$ for all $n \geq 1$. 
Let $u := \sum_{n \geq 0} u_n t^n$ and $v := \sum_{n \geq 1} v_n t^n$. One has of course
$\sum_{n \geq 1} a_n t^n = v^2 + t u^2$, and it is easy to see that $v^2 + v = 1 + u + t u^2$.

\medskip

Conversely, let $(u_n)_{n \geq 0}$ be any $2$-automatic sequence with $u_0 = 1$. Define the 
sequence $(a_n)_{n \geq 1}$ by $a_{2n+1} = u_n$ for all $n \geq 0$ and $a_{2n} = a_n + u_n$ 
for all $n \geq 1$. By construction the sequence $(a_n)_{n \geq 1}$ has a PLCP since 
$a_n + a_{2n} + a_{2n+1} = 0$ for all $n \geq 1$. Let us prove that it is
$2$-automatic. To make the manipulation of indices simpler, define $(\tilde{a}_n)_{n \geq 0}$ 
by $\tilde{a_0} = 0$ and $\tilde{a}_n = a_n$ for all $n \geq 1$. Furthermore let 
$(\delta_0(n))_{n \geq 0}$ be defined by $\delta_0(0) = 1$ and $\delta_0(n) = 0$ for all $n \geq 1$. 
The sequence $(a_n)_{n \geq 1}$ is $2$-automatic if and only if $(\tilde{a}_n)_{n \geq 1}$ is 
$2$-automatic. Since $(a_{2n+1})_{n \geq 0} = (\tilde{a}_{2n+1})_{n \geq 0}$ is $2$-automatic, its 
$2$-kernel is finite. Let ${\mathcal H}$ be the ${\mathbb F}_2$-vector space spanned by this $2$-kernel, the sequence $(\tilde{a}_n)_{n \geq 0}$ and the sequence $(\delta_0(n))_{n \geq 0}$. This
vector space has finite dimension, hence is finite. We will prove that the $2$-kernel of 
$(\tilde{a}_n)_{n \geq 0}$ is included in ${\mathcal H}$. It suffices to prove that ${\mathcal H}$
is stable by the operators $T_0$ and $T_1$. The images by $T_0$ and by $T_1$ of each element
in the $2$-kernel of $(\tilde{a}_{2n+1})_{n \geq 0}$, as well as the images of $(\delta_0(n))_{n \geq 0}$
clearly belong to ${\mathcal H}$. It thus suffices to prove that $T_0((\tilde{a}_n)_{n \geq 0})$ and
$T_1((\tilde{a}_n)_{n \geq 0})$ belong to ${\mathcal H}$. But we have
$$
T_0((\tilde{a}_n)_{n \geq 0}) = (\tilde{a}_{2n})_{n \geq 0} =
(\tilde{a}_n)_{n \geq 0} + (\tilde{a}_{2n+1})_{n \geq 0} + (\delta_0(n))_{n \geq 0}
$$
$$
\text{and} \ \ T_1((\tilde{a}_n)_{n \geq 0}) =  (\tilde{a}_{2n+1})_{n \geq 0}. \ \ \ \Box
$$

\begin{remark}\label{examples2}

We give examples of sequences that have PLCP and are $2$-automatic.

\begin{itemize}

\item[ ]$*$ It is easy to see that the $2$-kernel of each of the two Rueppel sequences is finite, 
hence these sequences are $2$-automatic.

\item[ ]$*$ The period-doubling sequence is an apwenian $2$-automatic sequence as shown
in \cite{Guo-Han-Wu}. Shifting the indices, this is the sequence $(z_n)_{n \geq 1}$ with values 
in ${\mathbb F}_2$ defined by $z_1 = 1$, $z_{2n} = 1+z_n$ for all $n \geq 1$, and 
$z_{2n+1} = 1$  for all $n \geq 1$. In particular $z_n + z_{2n} + z_{2n+1} = 0$ for all $n \geq 1$.

\item[ ]$*$ Define the sequence $(w_n)_n$ with values in ${\mathbb F}_2$  by $w_1 = 1$,
$w_{2n+1} = 1 + w_n$ for all $n \geq 1$, and $w_{2n} = 1$ for all $n \geq 1$. This sequence is
$2$-automatic (its $2$-kernel is finite). Clearly $w_n + w_{2n} + w_{2n+1} = 0$ for all $n \geq 1$.

\end{itemize}

Now we give examples of precise questions about (non-)conservation of ``regularity'' properties 
alluded to at the end of Remark~\ref{examples1} above. Namely: {\it If the sequence $\mathbf{b}$ 
is eventually periodic (i.e., periodic from some index on), algebraic, automatic, regular, morphic, ..., 
what can be said about $\mathbf{a} =\varphi_i(\mathbf{b})$?} Here are some answers.

\begin{itemize}

\item[ ]{$*$} If $\mathbf{b}$ is eventually periodic, then $\varphi_1(\mathbf{b})$ is quadratic. 
(This is well known and is the same as in the real case.)

\item[ ]{$*$} If $\mathbf{b}$ is the Thue-Morse sequence, then $\varphi_1(\mathbf{b})$ is 
$2$-automatic, and its associated formal power series is algebraic of degree $4$ 
\cite{Hu-Wei-Han2}. Other examples of automatic sequences $\mathbf{b}$, for which
$\varphi_1(\mathbf{b})$ is automatic can be found in \cite{Lasjaunias-Ruch1, Lasjaunias-Ruch2}.
It would be interesting to find an example of a sequence $\mathbf{b}$ such that 
$\varphi_1(\mathbf{b})$ is not automatic. The examples we have in mind go the other way round:
typically the Baum and Sweet sequence \cite{Baum-Sweet1} is a $2$-automatic sequence, but its 
associated power series has a continued fraction expansion ---with partial quotients of degree $1$ or
$2$--- that is not automatic \cite{Mkaouar}.

\item[ ]{$*$} As seen in Theorem~\ref{alg} above, $\varphi_2(\mathbf{b})$ is $2$-automatic if and 
only if $\mathbf{b}$ is $2$-automatic.

\item[ ]{$*$} We have that $\varphi_3(\mathbf{b})$ is $2$-automatic if and only if $\mathbf{b}$ 
is eventually periodic. Hint: the integers $n_h$ are exactly the integers with base-$2$ expansions 
$1$, $1b_0$, $1b_0b_1$, etc.; these expansions can be used to feed a direct automaton 
(i.e., an automaton reading the digits of the integers from left to right). J.~Shallit (private 
communication) gave us the more formal proof below.

\end{itemize}

\end{remark}
\begin{theorem}[{\rm J. Shallit}]\label{shallit} Let $(b_h)_h$ a sequence with values in $\{0,1\}$.
The generalized Rueppel sequence $(s_i)$ associated with $(b_h)_h$ is $2$-automatic if and only 
if $(b_h)$ is eventually periodic.
\end{theorem}

\begin{proof}
We take the notation of Remark~\ref{examples1} for the generalized Rueppel sequences where
$(c_h)_h$ is replaced with $(b_h)_h$. As noted above the set of all base-$2$ representations of the 
integers $n_h$ is the set $\{1b_0 b_1 b_2 \ldots b_i : i  \geq 0\}$. In other words these representations 
are the prefixes of the infinite word $1 b_0 b_1 b_2 \ldots$ We know that a binary sequence is 
$2$-automatic if and only if the base-$2$ representations of the indices where it is equal to $1$ form 
a regular set (see, e.g., \cite{Allouche-Shallit}). An easy classical result asserts that the set of all 
prefixes of an infinite word is regular if and only if that word is ultimately periodic. The result follows 
by combining these two claims. 
\end{proof}

\begin{remark}
The linear complexity of sequences generated by {\em cellular automata} was studied by several 
authors (see, e.g.,  \cite{FusterSabater} and the references therein), but this is of course a quite 
different story.
\end{remark}

\section{More Laurent series with partial quotients of degree one}

For a Laurent series, the equivalence between having partial quotients of degree $1$ and
having PLCP, namely properties (i) and (ii) in Theorem~\ref{equivalences} above for the case 
${\mathbb F}_2$, is actually true for any field (see \cite{Niederreiter2}). It is thus interesting
to find such series. Among many examples in the literature, we cite the following having all their
partial quotients of degree $1$. Mills and Robbins \cite{Mills-Robbins} give, for each prime 
$p \geq 3$, explicit examples of formal Laurent series with coefficients in ${\mathbb F}_p$ that are 
algebraic over ${\mathbb F}_p(X)$ and have all their partial quotients of degree $1$. Examples
in characteristic $2$ can be found in \cite[p.~290]{Thakur}. Quartic power series in 
${\mathbb F}_3((X^{-1}))$ are given by Lasjaunias \cite{Lasjaunias2000}. More general series 
are studied in \cite{Lasjaunias-Ruch1} and \cite{Lasjaunias-Ruch2} (where the name {\em flat series} 
is used for algebraic series having all partial quotients of degree $1$). Hyperquadratic power series 
in ${\mathbb F}_3((X^{-1}))$ are given in \cite{GomezPerez-Lasjaunias}, while \cite{Lasjaunias-Yao} 
gives a large family of examples in odd characteristic and \cite{Lasjaunias2017} addresses the 
case of characteristic $2$.

\medskip

Addressing the previous examples uses their {\em hyperquadraticity}: an irrational element 
$\alpha$ in ${\mathbb F}_q((X^{-1}))$ is called hyperquadratic if it satisfies an equation 
$\alpha = (A\alpha^r+B)/(C\alpha^r+D)$, where $r$ is a power of the characteristic of 
${\mathbb F}_q$ and the coefficients $A, B, C, D$ belong to ${\mathbb F}_q[T]$. For
examples of a different nature, van der Poorten and coauthors studied series related to the 
{\em folding lemma} for continued fractions and to paperfolding or to playing between finite fields 
and rational numbers, see \cite{VdP-Shallit} and \cite{A-MF-VdP}.
We cannot resist to quote a remark of A. van der Poorten in \cite{A-MF-VdP}: 
{\it By the way, it is rather easy to see that, in the `generic case', an infinite series has all its partial
quotients linear --- though it is debatable whether a series with coefficient $0$ or $1$ is `generic'.}
The results in \cite{A-MF-VdP} were generalized in \cite{Cantor}, where a formal Laurent series is 
called {\em normal} if its continued fraction is not finite and all of the partial quotients, except 
perhaps the first, have degree $1$. About continued fractions of Laurent series, and, in particular,
{\em normal} series, one should certainly read the nice paper \cite{VdPPortland}. 
Interestingly enough a variation on the Thue-Morse power series, namely the Laurent series
$t \prod_{k \geq 0} (1 - t^{-2^k})$ has all its partial quotients of degree $1$, see \cite[Prop.~3.2]{BZ}.

Note that the seemingly simpler study of rational functions $p/q$ having all partial quotients of 
degree one in their continued faction expansion, is far from being straightforward, see, e.g., 
\cite{Mesirov-Sweet, Niederreiter-monatshefte, Wang, Blackburn, Lauder1998, Friesen}; 
note that in \cite{Blackburn}, on p.~104, there is a link, for sequences with values in ${\mathbb F}_2$, 
between certain Hankel determinants being different from $0$ and the relations 
$s_n + s_{2n} + s_{2n+1} = 0 \bmod 2$ for certain values of $n$.

\begin{remark}
We have seen that several terms are used to name Laurent series with all their partial quotients of 
degree $1$, or the sequences of coefficients of these series: namely PLCP, flat, normal. Another
term is used, e.g., in \cite{Blackburn, Lauder1998}: the {\em orthogonal multiplicity} of a monic 
polynomial $g$ is the number of polynomials $f$, coprime to $g$ and of degree less than the 
degree of $g$, such that all the partial quotients of the continued fraction expansion of $f/g$ are of 
degree $1$: hence saying that there exists $f$ such that $f/g$ has all its partial quotients of degree 
$1$, is the same as saying that $g$ has positive orthogonal multiplicity. Let us also point out the
name {\em badly approximable} for a series whose continued fraction expansion has partial 
quotients with bounded degree. 
\end{remark}

\section{Miscellanea}
In this section we propose a few questions that we could either not answer, or for which
we only have partial results.

\begin{itemize}

\item[ ]{$*$} We have seen conditions on a PLCP/apwenian sequence to be $2$-automatic. The 
question of whether an apwenian sequence can be an iterative fixed point of a $d$-substitution
was addressed in \cite{Guo-Han-Wu}: {\it the only $(0,1)$-apwenian sequence which is a fixed 
point of some $d$-substitution is the period doubling sequence}. More generally, are there
PLCP/apwenian sequences that are $d$-automatic for some $d$ not a power of $2$, or even 
morphic? Note that the authors of \cite{Guo-Han-Wu} conjecture that the {\it fixed points of 
substitutions of non-constant length on $\{0, 1\}$ cannot be apwenian}.

\item[ ]{$*$} Is there a ``combinatorial'' condition for PLCP on larger alphabets similar to the relation
$s_{2n+1} + s_{2n} + s_n = 0$? And what would be apwenian sequences in this context?
Note that a relation of the same kind appears in \cite[p.~270]{Lauder1999}.

\item[ ]{$*$} Is the sequence of Hankel determinants of a $d$-automatic sequence a $d$-regular
sequence or, once reduced modulo some $k$, a $d$-automatic sequence? Some positive results
in this direction can be found in \cite{Hu-Wei-Han}.

\item[ ]{$*$} Can the previous question be addressed by studying in detail the way of transforming
a Stieltjes, or a Jacobi, or a Hankel, continued fraction into a usual continued fraction?

\item[ ]{$*$} The linear complexity profile of some automatic sequences was computed or evaluated
in \cite{Merai-Winterhof}. In particular, it is proved that several $q$-automatic sequences over 
${\mathbb F}_q$, which are thus somehow ``predictable'', have their linear complexity $L(n)$ of 
order of magnitude $n$. For example, {\it the linear complexity profile of the Thue-Morse sequence
is the sequence $(2 \lfloor \frac{n+2}{4}\rfloor)_{n \geq 1}$ and the linear complexity profile of the 
period-doubling sequence is the sequence $(\lfloor \frac{n+1}{2} \rfloor)_{n \geq 1}$} (by the way
this sequence is the only binary sequence having both perfect linear complexity profile and 
{\em perfect lattice profile}, see \cite{Dorfer-Meidl-Winterhof} for more details). 
For which other automatic or morphic sequences is it possible to compute exactly their linear 
complexity profile?

\item[ ]$(*)$ How does linear complexity of a sequence compare with other ``complexities''?
The reader can, in particular, consult the following three papers: \cite{Shallit} (where {\it inter alia} 
a measure called {\em rationality} that generalizes linear complexity is introduced), \cite{Diem} 
(where a discussion about terminology takes place at the end of the paper, namely about the terms
``perfect linear complexity profile'', ``almost perfect linear complexity profile'', ``$d$-almost perfect 
complexity'' and ``$d$-perfect''), and \cite{GMN}.

\end{itemize}

\noindent
{\bf Acknowledgments} We thank J. Shallit for having provided a proof of Theorem~\ref{shallit}.
The first author would like to thank A. Lasjaunias and J.-Y. Yao for interesting discussions.

\end{document}